\documentclass[12pt,a4paper]{amsart}
\usepackage{amsfonts}
\usepackage[active]{srcltx}

\usepackage{enumerate}
\usepackage{amsmath,amssymb,xspace,amsthm}

\newtheorem{theorem}{Theorem}[section]

\newtheorem{lemma}[theorem]{Lemma}

\newtheorem{definition}[theorem]{Definition}
\numberwithin{equation}{section} \theoremstyle{definition}

\def\Der{\operatorname{Der}}
\def\span{\operatorname{span}}

\newcommand{\C}{\ensuremath{\mathbb C}\xspace}

\renewcommand{\a}{\ensuremath{\alpha}}
\renewcommand{\l}{\ensuremath{\lambda}}

\newcommand{\Z}{\ensuremath{\mathbb{Z}}\xspace}
\newcommand{\N}{\ensuremath{\mathbb{N}}\xspace}

\newcommand{\A}{\ensuremath{\mathcal{A}}\xspace}
\newcommand{\W}{\ensuremath{\mathcal{W}}\xspace}

\newcommand{\ad}{\operatorname{ad}\xspace}

\renewcommand{\phi}{\varphi}

\renewcommand{\leq}{\leqslant}
\renewcommand{\geq}{\geqslant}

\def\mh{\mathfrak{h}}

\def\sl{\mathfrak{sl}}
\def\gl{\mathfrak{gl}}

\def\si{\sigma}

\def\l{\lambda}

\def\rad{\text{Rad}}
\def\ad{\text{ad}}
\def\span{\text{span}}

\def\Der{\text{Der}}

\begin{document}
\title[Irreducible weight modules]{Classification of irreducible bounded weight modules over the derivation Lie algebras of quantum tori}
\author{Genqiang Liu and Kaiming Zhao}
\date{Jun, 2015.}
\maketitle

\begin{abstract}
Let $d>1$ be an integer, $q=(q_{ij})_{d\times d}$ be a $d\times d$ complex matrix
satisfying $q_{_{ii}}=1, q_{_{ij}}=q_{_{ji}}^{-1}$ with all $q_{ij}$ being roots of unity.  Let $\C_q$ be the
rational quantum torus algebra associated to $q$, and  $\Der(\C_q)$ its
derivation Lie algebra. In this paper, we give a complete classification of irreducible  bounded weight modules
over $\Der(\C_q)$. They turn out to be  irreducible sub-quotients of  $\Der(\C_q)$-module $\mathcal{V}^\alpha(V,W)$ for a
finite dimensional irreducible $\gl_d$-module $V$, a finite
dimensional $\Gamma$-graded-irreducible
 $\gl_N$-module $W$, and $\alpha\in \mathbb{C}^d$.

\end{abstract}


\vskip 5pt
\noindent
{\em 2000  Math. Subj. Class.:}
17B10, 17B20, 17B65, 17B66, 17B68

\vskip 10pt

\section{Introduction}

Let  $d>1$ be an integer, $q=(q_{ij})_{d\times d}$ be a $d\times d$ complex matrix
satisfying $q_{_{ii}}=1, q_{_{ij}}=q_{_{ji}}^{-1}$ with all $q_{ij}$
being roots of unity.  In the present paper, we consider the
rational quantum torus algebra $\C_q$ associated to $q$, and its
derivation algebra $\Der(\C_q)$. The algebra $\C_q$  is an important
algebra, since it is the coordinate algebra of a large class of
extended affine Lie algebras (See \cite{BGK}) and shows
 up in the theory of noncommutative geometry (See \cite{BVF}).
When all $q_{ij}=1$, the algebra $\Der(\C_q)$ is the classical
Witt algebra $\W_d$, i.e., the derivation algebra of the
Laurent polynomial algebra $\A=\C[x_1^{\pm1},x_2^{\pm1}, . . .,
x_d^{\pm1}]$, see \cite{RSS}, which is also known as the Lie algebra of vector fields on a $d$-dimensional torus.

The representation theory of Witt algebras was
studied by many mathematicians and physicists for the last couple of decades, see \cite{B, E1, E2, GLZ,L3, L4, L5,
MZ,Z1,Z2}. In 1986,  Shen defined a class of modules $F^\alpha_b(V)=V\otimes \A$
over the Witt algebra $\mathcal{W}_d$ for  $\a\in\C^d$
and an irreducible module $ V$ over the general linear Lie algebra
$\gl_d$ on which the identity matrix acts as multiplication by a  complex number $b$, see \cite{Sh}, which were also given by Larsson in 1992,
see \cite{L3}. In 1996, Eswara Rao \cite{E1} determined necessary and
sufficient conditions for   these modules to be irreducible when $V$
is finite dimensional, see   \cite{GZ} for a simplified proof. When $V$ is infinite dimensional, $F^\alpha_b(V)$ is always
irreducible, see \cite{LZ2}.

 Very recently Billig and Futorny
\cite{BF2} gave a complete classification of all irreducible weight modules over $\W_d$ with
finite dimensional weight spaces. Based on \cite[Theorem 3.1]{MZ} they actually showed that any irreducible
bounded  weight modules over $\W_d$ is isomorphic to some irreducible subquotient  of $F^\alpha_b(V)$.
To achieve this result, they introduced a   powerful technique: any bounded
weight $\W_d$-module $M$ is a $\W_d$-quotient module of
an $\A\W_d$-module $\widehat{M}$, a module both for the Lie algebra $\W_d$ and the associative algebra $\A$ with two
structures being compatible. Here $\widehat{M}$ is called the $\A$-cover of $M$, which
is in fact an $\A\W_d$-quotient module of the $\A\W_d$-module $\W_d\otimes M$.
 Thus they reduced the classification of irreducible bounded $\W_d$-modules to the classification
of irreducible bounded $\A\W_d$-modules. Using the  classification
of irreducible bounded $\A\W_d$-modules in \cite{E2,B}, they classified  all irreducible bounded weight modules over $\W_d$.

Lin and Tan defined in \cite{LT} a class of uniformly bounded irreducible
weight modules over $\Der(\C_q)$, which generalized the construction given
by Shen. These modules were clearly characterized in \cite{LZ3}. But these modules can not exhaust all simple bounded weight modules over $\Der(\C_q)$,
 since
 a bigger class of simple modules $\mathcal{V}^\alpha(V,W)$ were  constructed in \cite{LZ1}, which further generalized Shen's modules.   See (2.5).
 Moreover we showed in \cite{LZ1} that
any irreducible $Z\mathcal{D}$-weight module (similar to the notion of $\A\W_d$-modules, see Definition 2.2)
with finite dimensional weight spaces is isomorphic to some  $\mathcal{V}^\alpha(V,W)$ for a
finite dimensional irreducible $\gl_d$-module $V$, a finite
dimensional $\Gamma$-graded-irreducible
 $\gl_N$-module $W$, and $\alpha\in \mathbb{C}^d$, where $Z:=Z(\C_q)$ is the center of
$\C_q$ and $\mathcal{D}:=\Der(\C_q)$.

In the present paper, we consider  irreducible bounded weight $\Der(\C_q)$-modules. For an irreducible bounded weight $\Der(\C_q)$-module $M$, we construct a $Z\mathcal{D}$-module $\widehat{M}$ which is called the $Z\mathcal{D}$-cover of $M$.
  The ideal of the $Z\mathcal{D}$-cover stems from \cite{BF2}. Here the $Z\mathcal{D}$-cover $\widehat{M}$ is different
   from the $\A\W_d$-cover $\W_d\otimes M$ in \cite{BF2}, since $\W_d\otimes M$
  is no longer a $Z\mathcal{D}$-module in our case. Now we define the $Z\mathcal{D}$-cover $\widehat{M}$ as
  a $Z\mathcal{D}$-quotient module of $\C_q'\otimes M$, see Definition 3.4.
  Using this technique, we prove that any irreducible bounded $\Der(\C_q)$-weight module
  is isomorphic to some irreducible sub-quotient of  $\mathcal{V}^\alpha(V,W)$ for a
finite dimensional irreducible $\gl_d$-module $V$, a finite
dimensional $\Gamma$-graded-irreducible
 $\gl_N$-module $W$, and $\alpha\in \mathbb{C}^d$. See Theorem 2.5.

Throughout this paper we denote by $\mathbb{Z}$, $\mathbb{Z}_+$, $\mathbb{N}$,
$\mathbb{Q}$ and $\mathbb{C}$ the sets of  all integers, nonnegative
integers, positive integers, rational numbers and complex numbers,
respectively. We use $E_{ij}$ to denote the matrix with a $1$
in the $(i, j)$ position and zeros elsewhere.

\section{Notation and the main result}

In this section we will collect notation and related results, then state our main theorem.

We fix a positive integer $d>1$. Denote  vector space  of $d\times 1$ matrices by $\mathbb{C}^d$.
Denote its standard basis by $\{e_1,e_2,...,e_d\}$. Let
$(\,\cdot\,|\, \cdot\, )$ be the standard symmetric bilinear form
such that $(u | v)=u^Tv\in\mathbb{C}$, where $u^T$ is the matrix
transpose of $u$.

 Let $q=(q_{_{ij}})_{i,j=1}^d$ be a $d\times d$ matrix
over $\C$ satisfying $q_{_{ii}}=1, q_{_{ij}}=q_{_{ji}}^{-1},$ where
$q_{ij}$ are roots of unity for all $1\leq i,j\leq d$. We will call
such a matrix $q$ {\it rational}.
\begin{definition}The rational quantum torus
$\mathbb{C}_q$ is the unital associative algebra over
$\C$ generated by $t_1^{\pm 1}, \ldots, t_d^{\pm 1}$ and subject to
the defining relations $t_it_j=q_{ij}t_jt_i,$
$t_it_i^{-1}=t_i^{-1}t_i=1$ for all $1\leq i,j\leq d$.
\end{definition}
For convenience, denote $t^n=t_1^{n_1}t_2^{n_2}\cdots t_d^{n_d}$ for any
$n=(n_1,\cdots, n_d)^T\in\mathbb{Z}^{d}$.
For any $n,m\in\mathbb{Z}^{d}$, we define the functions $\si(n,m)$ and $f(n,m)$ by $$t^n
t^m=\si(n,m)t^{n+m}, \,\,t^n t^m=f(n,m)t^m t^n.$$ It is well-known
that
$$\si(n,m)=\prod_{1\le i<j\le d}q_{ji}^{n_j
m_i},\,\,\,f(n, m)=\prod_{i,j=1}^dq_{ji}^{n_j m_i},$$ and
$f(n,m)=\si(n,m)\si(m,n)^{-1}$, see \cite{BGK}. We also define $$\rad(f)=\{n\in
\mathbb{Z}^{d}\,|\,f(n,\mathbb{Z}^{d})=1\},\,\,\Gamma=\mathbb{Z}^{d}/\rad(f)
.$$
Clearly, the center $Z(\C_q)$ of $\C_q$ is spanned by $t^r$ for $r\in\rad(f)$.

From the results in \cite{N},
 up to an isomorphism of  $\C_q$,  we
may assume that $q_{2i,2i-1}=q_i, q_{2i-1,2i}=q_i^{-1},$ for $1\leq
i\leq z$,  and other  entries of $q$ are all $1$, where
$z\in\mathbb{N}$ with $2z\leq d$ and  with the orders $k_i$  of
$q_i,1\leq i\leq z$ as roots of unity satisfy  $k_{i+1}|k_{i}, 1\leq
i< z$. For an integer $l\in\{1,\dots,d\}$, let \begin{equation}\label{2.1}
  \xi_l=
  \begin{cases}
   k_ie_{2i-1},& {\text {if} }\ l=2i-1\leq 2z,\\
k_ie_{2i},& {\text {if} }\ l=2i\leq 2z,\\
e_{l},& {\text {if} }\ l>2z.\\
\end{cases}\end{equation}
Then $\{\xi_1,\dots,\xi_d\}$ is a $\Z$-basis of the subgroup $\rad(f)$.

Throughout the present  paper, we assume that   $q$ is of the above simple form. Then we see that
$\sigma(r,n)=\sigma(n,r)=1$ (i.e.,
$t^nt^r=t^{n+r}$) for all $r\in\rad(f)$ and $n\in\Z^d$.
In this case, we know that $$\Gamma =\oplus_{i=1}^z \left
(\Z/(k_i\Z))\oplus (\Z/(k_i\Z))\right).$$


Let $\Der(\mathbb{C}_q)$ be the derivation Lie algebra of $\C_q$. Let $\Der(\mathbb{C}_q)_n$
be the set of homogeneous elements of $\Der(\mathbb{C}_q)$
with degree $n\in\Z^d$. Then from Lemma 2.48 in [BGK], we have
$$\Der(\mathbb{C}_q)=\bigoplus_{n\in\mathbb{Z}^d}\Der(\mathbb{C}_q)_n,\,\,
  \Der(\mathbb{C}_q)_n=
  \begin{cases}
   \mathbb{C}{\ad}(t^n),& {\text {if} }\ n\not\in {\rad(f)},\\
\bigoplus_{i=1}^d\mathbb{C}t^n\partial_i,& {\text {if} }\ n\in
   {\rad}(f),
\end{cases}$$
where $\partial_i$ is the degree derivation defined by $\partial_i(t^n)=n_it^n$ for any $n\in\mathbb{Z}^d$. We
will simply denote ad$(t^n)$ in $\Der(\mathbb{C}_q)$ by $t^n$ for $n\not\in
{\rad(f)}$.

For $n\in\rad(f),u \in\mathbb{C}^d$, we denote
$D(u,n)=t^n\sum_{i=1}^du_i\partial_i$. The Lie bracket of $\Der(\mathbb{C}_q)$ is given
by:
\begin{itemize}
\item[(1)]$[t^s, t^{s'}]=(\sigma(s, s')-\sigma(s', s))t^{s+s'}$;
\item[(2)]$[D(u, r), t^s]=(u | s) t^{r+s}$;
\item[(3)]$[D(u, r), D(u', r')]=D(w, r+r')$,
\end{itemize}
where $w =(u | r')u'-(u' \,|\, r)u$, $s, s'\in \mathbb{Z}^d\setminus \rad(f)$, $r,r'\in\rad(f)$, and we have used that $\sigma(r, s)=\sigma(r,r')=1$.

We can  see that $\mh:=\span\{D(u, 0)\mid u\in\mathbb{C}^d\}$
is the Cartan  subalgebra (the maximal toral subalgebra) of $\Der(\C_q)$.  Moreover the subalgebra of $\Der(\C_q)$
spanned by $\{ t^s\,|\,s\in\Z^d\backslash \rad(f)\}$ is isomorphic to the derived algebra $\C_q':=[\mathbb{C}_q,\C_q]$ of $\C_q$.
 Let
$$\W_d=\span\{D(u,r)\mid r\in\rad(f), u\in\mathbb{C}^d\}$$ which is indeed isomorphic to the
classical Witt algebra. Note that the algebra $\Der(\C_q)$ has a nature structure of  $Z(\C_q)$-module, i.e.,
$$t^r\cdot t^s=t^{s+r},\ \ t^r\cdot D(u,r')=D(u,r+r'),$$ where
$r,r'\in\rad(f), s\in\Z^d\setminus\rad(f), u\in \C^d$.

 A $\Der(\C_q)$-module $V$ is called a {\em weight} module provided
that the action of $\mh$ on $V$ is diagonalizable.  For
any weight module $V$ we have the weight space decomposition
\begin{equation}
V=\bigoplus_{\lambda\in \mh^*}V_{\lambda},
\end{equation}
where
$\mh^*=\mathrm{Hom}_{\mathbb{C}}(\mh,\mathbb{C})$
and
\begin{displaymath}
V_{\lambda}=\{v\in V\mid \partial v=\lambda(\partial)v \text{ for all
}\partial\in \mh\}.
\end{displaymath}
The space $V_{\lambda}$ is called the {\em weight space}
corresponding to the {\em weight} $\lambda$.  If there is an integer $k\in\N$ such that $\dim_{\mathbb{C}}V_{\lambda}<k$
for all $\lambda\in
\mh^*$, the weight module $V$ is called a {\em bounded} weight module. The following notion is important to our later arguments.


\begin{definition}A  $Z\mathcal{D}$-module $V$ is  a  module both
for the Lie algebra $\Der(\C_q)$ and the commutative associative algebra $Z(\C_q)$, with these two structures
being compatible:
\begin{equation}[D(u,r),t^{r'}]v=D(u,r)t^{r'}v-t^{r'}D(u,r)v,
\end{equation}
\begin{equation}t^s t^rv=t^rt^sv,
\end{equation}for any $r,r'\in \rad(f), s\not\in\rad(f), v\in V$.
\end{definition}
Clearly $\C_q'$ is a $Z\mathcal{D}$-module under the adjoint action of $\Der(\C_q)$ and  the action of $Z(\C_q)$ defined as follows:
$$ t^r t^n=t^{n+r}, \ \ r\in\rad(f), \ \ n\not\in\rad(f).$$

In \cite{LZ1}, a class of $Z\mathcal{D}$-modules was constructed. Next, we will recall these modules.
First, we recall the twisted loop algebra realization of $\C_q$.

Let $\mathcal{I}=\span\{t^{n+r}-t^n\,|\,n\in \mathbb{Z}^d, r\in\rad
(f)\}$ which is an ideal of the associative algebra $\mathbb{C}_q$.
Then from \cite{N} and \cite{Z2} we know that
$$\mathbb{C}_q/\mathcal{I}\simeq\otimes_{i=1}^z\gl_{k_i}\simeq \gl_N$$
as associative algebras with $N=\prod\limits_{i=1}^z k_i$. It is
well known that $\gl_{k_i}, 1\leq i\leq z$, as the associative
algebra $M_{k_i}(\C)$, is generated by $X_{2i-1}, X_{2i}$ with
$$X_{2i-1}=E_{1,1}+q_iE_{2,2}+\cdots+q_i^{k_i-1}E_{k_i,k_i},$$ $$X_{2i}=E_{1,2}+E_{2,3}+\cdots
+E_{k_i-1,k_i}+E_{k_i,1},$$ which satisfy
$X_{2i}^{k_i}=X_{2i-1}^{k_i}=1, X_{2i}X_{2i-1}=q_iX_{2i-1}X_{2i}$.
We denote $\otimes_{i=1}^z X_{2i-1}^{n_{2i-1}}X_{2i}^{n_{2i}}$ by
$X^n$ for each $n\in\mathbb{Z}^d$. Identifying  $\gl_N$ with
$\bigotimes\limits_{i=1}^z\gl_{k_i}$ as  associative algebras,
$\gl_N$ is spanned by $X^n, n\in\mathbb{Z}^{d}$ and $X^r$ equals to
the identity matrix $E$ in $\gl_N$ for each $r\in\rad(f)$.

\begin{lemma} \emph{(See \cite{ABFP})} As associative algebras, $$\mathbb{C}_q\cong\bigoplus\limits_{n\in \mathbb{Z}^d}
(\mathbb{C}X^n\otimes x^n),$$ where  the right hand side is a
$\mathbb{Z}^d$-graded subalgebra of
$\gl_N\otimes \A$.
\end{lemma}

 Clearly, $\gl_N$ is a
$\Gamma$-graded Lie algebra with the gradation
$$\gl_N=\bigoplus_{{\bar n}\in\Gamma}(\gl_N)_{\bar n},$$ where
$(\gl_N)_{\bar n}=\C X^n$.

A module $W$ over the Lie algebra $\gl_N$ is called a
$\Gamma$-graded $\gl_N$-module if $W$ has a subspace decomposition
$W=\bigoplus_{{\bar n}\in\Gamma}W_{\bar n}$ such that
$(\gl_N)_{\bar m}W_{\bar n}\subset W_{\bar m+\bar n}$ for all
$m,n\in \Z^d$. A $\Gamma$-graded $\gl_N$-module  $W$ is  $\Gamma$-graded-irreducible if
it has no nonzero proper  $\Gamma$-graded submodules.
We remark that all finite dimensional  $\Gamma$-graded $\gl_N$-modules were classified in \cite{EK}.

For any irreducible finite dimensional $\gl_d$-module $V$, any $\Gamma$-graded-irreducible $\gl_N$-module
$W=\bigoplus_{{\bar n}\in\Gamma}W_{\bar n}$ with identity action of identity
matrix $E$ in $\gl_N$, and any  $\alpha\in
\mathbb{C}^d$, let
\begin{equation}\mathcal{V}^\alpha(V,W)=\bigoplus_{n\in\Z^d}(V\otimes W_{\bar n}\otimes t^n).\end{equation}
 Then $\mathcal{V}^\alpha(V,W)$ becomes a $Z\mathcal{D}$-module
if we define the following actions
\begin{itemize}
  \item[(1)] $t^s(v\otimes w_{\bar n}\otimes t^n)=v\otimes (X^sw_{\bar n})\otimes t^{n+s}$;
  \item[(2)] $D(u,r)(v\otimes w_{\bar n}\otimes t^n)=\Big((u\,|\,n+\alpha)v+ (ru^T)v\Big)\otimes w_{\bar n}
\otimes t^{n+r}$,
\end{itemize}
where $u\in\mathbb{C}^d$, $v\in V$, $w_{\bar n}\in W_{\bar n}$ and $
r\in \rad(f), s\in  \mathbb{Z}^d, X^s\in \gl_N$.

In \cite{LZ1}, all irreducible $Z\mathcal{D}$-modules with finite dimensional
weight spaces are proved to be  of the form $\mathcal{V}^\alpha(V,W)$.
Restricted on $\Der(\C_q)$, $\mathcal{V}^\alpha(V,W)$ is not necessarily irreducible.
The following result easily follows from \cite{E1} and \cite{GZ}, which gives all irreducible subquotients of the $\Der(\C_q)$-module $\mathcal{V}^\alpha(V,W)$.

\begin{lemma} The $\Der(\C_q)$-module $\mathcal{V}^\alpha(V,W)$ is reducible
if and only if $\dim W=1$ and one of the following holds
\itemize\item[(a).]
  the highest weight of $V$ is the fundament weight
$\omega_k$ of $\sl_d$ and $b=k$, where $k\in\Z$ with $1\leq k\leq
d-1$;
\item[(b).] $\dim V=1$, $\a\in\Z^d$ and $b\in\{0, d\}$.\enditemize\end{lemma}

We can easily see that when the $\Der(\C_q)$-module $\mathcal{V}^\alpha(V,W)$ is reducible it has a unique nonzero proper submodule.

In the present paper,  we will reduce the classification irreducible uniformly bounded modules over $\Der(\C_q)$ to the
classification of irreducible $Z\mathcal{D}$-modules, that is,  we will obtain the following main result.

\begin{theorem}\label{t4.4} Let $d>1$ be an integer, $q=(q_{ij})_{d\times d}$ be a $d\times d$ complex matrix which is rational.
Let $M$ be an irreducible   bounded  weight $\Der(\C_q)$-module. Then there exist a
finite dimensional irreducible $\gl_d$-module $V$, a finite
dimensional $\Gamma$-graded-irreducible
 $\gl_N$-module $W$, and $\alpha\in \mathbb{C}^d$ such that
$M$ is isomorphic to some irreducible sub-quotient of  $\mathcal{V}^\alpha(V,W)$.
\end{theorem}

\section{Proof of Theorem 2.5}

In this section we will prove Theorem 2.5.

Let $M$ be an irreducible bounded weight $\Der(\C_q)$-module. The irreducibility of $M$ implies that there is an $\a\in\C^d$ such that
$M=\oplus_{n\in\Z^d}M_{\alpha+n}$, where $$M_{\alpha+n}=\{v\in M \mid \partial_i(v)=(\a_i+n_i)v,\ 1\leq i\leq d\}.$$
In \cite{BF2}, in order to define  the $\A\W_d$-cover of $M$,
they considered the the tensor product $\W_d\otimes M$ of the adjoint module and $M$. In our case, the module $\W_d\otimes M$ is still an $\A\W_d$-module, unfortunately is no longer a $Z\mathcal{D}$-module.
Now we  turn to  the tensor product  $\C_q' \otimes M$ of the $\Der(\C_q)$-modules $\C_q'$ and $M$,
since $\C_q'$ itself is a $Z\mathcal{D}$-module.

\begin{lemma} The space $\C_q' \otimes M$ is a $Z\mathcal{D}$ module if we define the action of $Z(\C_q)$ by
\begin{equation}\label{3.1} t^r(t^n\otimes w)=t^{n+r}\otimes w,\end{equation}
where $r\in\rad(f), n\not\in\rad(f), w\in M$.

\end{lemma}
\begin{proof}
For any $u\in\C^d, m,r\in\rad(f), n, s\not\in\rad(f)$ and $w\in M$,  we have that
$$\aligned
 D(u,m)t^r&(t^n\otimes w)-t^rD(u,m)(t^n\otimes w)\\
=&\ (u\mid r+n)t^{n+m+r}\otimes w+t^{n+r}\otimes D(u,m)w\\
&\ -(u\mid n)t^{n+m+r}\otimes w-t^{n+r}\otimes D(u,m)w \\
=&(u\mid r)t^{n+m+r}\otimes w=[D(u,m),t^r](t^n\otimes w),
\endaligned $$
and
$$\aligned   t^st^r(t^n\otimes w)
=&\ (\sigma(s,n+r)-\sigma(n+r,s))t^{n+s+r}\otimes w+t^{n+r}\otimes t^s w\\
=&\ (\sigma(s,n)-\sigma(n,s))t^{n+s+r}\otimes w+t^{n+r}\otimes t^s w\\
=&\ t^rt^s(t^n\otimes w).
\endaligned $$
In the second equality, we have used the fact that
$\sigma(r,s)=\sigma(s,r)=1$.

So the action of  $\Der(\C_q)$ and $\Z(\C_q)$ is compatible, hence $\C_q'\otimes M$ is a $Z\mathcal{D}$ module.
\end{proof}

Define the linear map
$$\pi:\C_q'\otimes M \rightarrow M$$
 by $\pi(y\otimes w)=yw$ for $y\in \C_q', w\in M$.
\begin{lemma} The map $ \pi$ is a $\Der(\C_q)$-module homomorphism. When $\C_q' M\neq 0$, $\pi$ is surjective.
\end{lemma}
\begin{proof} For all $n\not\in\rad(f), r\in\rad(f), w\in M$, we have that $$\aligned
\pi(D(u,r)(t^n\otimes w))
=&\ [D(u,r),t^n]w+t^nD(u,r)w\\
=&\ D(u,r)t^nw=D(u,r)\pi(t^n\otimes w).
\endaligned $$
So $\pi$ is a $\Der(\C_q)$-module homomorphism. It is easy to see that
$\C_q' M$ is a submodule of $M$. Then the irreducibility of $M$ implies that $\pi$ is surjective.
\end{proof}

Let $J$ be the subspace of $\C_q'\otimes M$ spanned by the set
$$\{ \sum_{n\in I} t^n \otimes v_n \mid n\not\in \rad(f), v_n\in M, \sum_{n\in I} t^{n+\gamma}v_n=0, \text{for all}\ \gamma\in\rad(f) \}.$$
Clearly, $J\subset \ker(\pi)$.

\begin{lemma} The subspace $J$ is a $Z\mathcal{D}$-submodule of $\C_q'\otimes M$.
\end{lemma}

\begin{proof} Let $\eta=\sum_{n\in I} t^n \otimes v_n\in J$, where $I\subset \Z^d\setminus\rad(f)$ is a finite subset.
Then $$\sum_{n\in I} t^{n+r}v_n=0, \text{ for all } r\in\rad(f).$$
To show that $J$ is a $Z\mathcal{D}$-submodule, we only need to show that $$t^{r'}\eta,\ D(u,r')\eta,\ t^s\eta\in J,
\text{ for any } r'\in \rad(f),\  s\not\in\rad(f).$$ From  $\sum_{n\in I} t^{n+r+r'}v_n=0$, we see that
$$\aligned\sum_{n} (u\mid &n)  t^{n+r'+r}v_n+\sum_{n\in I} t^{n+r}D(u,r')v_n\\
=& \sum_{n\in I} (u\mid n)t^{n+r'+r}v_n+\sum_{n\in I} [t^{n+r},D(u,r')]v_n+D(u,r')\sum_{n\in I} t^{n+r}v_n\\
=& \sum_{n\in I} (u\mid n)t^{n+r'+r}v_n-\sum_{n\in I} (u\mid n+r)t^{n+r'+r}v_n\\
=& -(u\mid r)\sum_{n\in I} t^{n+r'+r}v_n=0.
\endaligned $$
Note that  $$t^{r'}\eta=\sum_{n\in I} t^{n+r'}\otimes v_n,$$
 $$D(u,r')\eta=\sum_{n\in I} (u\mid n)t^{n+r'}\otimes v_n+\sum_{n\in I} t^{n}\otimes D(u,r')v_n.$$
So $t^{r'}\eta, D(u,r')\eta\in J$.

 From
$$\aligned \sum_{n\in I} (\sigma(s,n)& -\sigma(n,s))t^{n+s+r}v_n+\sum_{n\in I} t^{n+r}t^sv_n\\
=& \sum_{n\in I} [t^s, t^{n+r}]v_n+\sum_{n\in I} [t^{n+r},t^s]v_n+t^s\sum_{n} t^{n+r}v_n\\
=& \ 0,
\endaligned $$ and $$t^s\eta =\sum_{n\in I} (\sigma(s,n)-\sigma(n,s))t^{n+s}\otimes v_n+\sum_{n\in I} t^{n}\otimes t^sv_n,$$
we see that $t^s\eta\in J$. So $J$ is a $Z\mathcal{D}$-submodule.
\end{proof}

\begin{definition} The $Z\mathcal{D}$-module $\widehat{M}:=\left(\C_q'\otimes M\right)/ J$ is called the $Z\mathcal{D}$-cover of $M$.
\end{definition}

Since $J\subset \ker(\pi)$, $\pi$ induces an epimorphism from $\widehat{M}$ to $M$ which is stilled
denoted by $\pi$. For $t^n\otimes v\in \C_q'\otimes M$, denote the its image in $\widehat{M}$ by
$\psi(t^n, v)$. The next key step is to show that $\widehat{M}$ is a bounded weight module. We will use the
{solenoidal} Lie algebra (or called the centerless higher rank Virasoro algebra) as an auxiliary instrument.

Recall from \cite{BF2} that a vector $u\in\C^d$ is generic if $(u|r)\neq 0$ for any $r\in\Z^d\backslash\{0\}$.
 For a generic   vector $u\in\C^d$, let $$e_r=D(u,r)\text{ for }r\in\rad(f).$$ The subalgebra $W_u$ of
$\Der(\C_q)$ spanned by $e_r$, $r\in\rad(f)$ is  a {solenoidal} Lie algebra.

From now on, we fix a  generic   vector $u\in\C^d$. It is easy to see that the Lie bracket of $W_u$ is given by
\begin{equation}\label{3.2} [e_r,e_{r'}]=(u\mid r'-r)e_{r+r'}\ \ \ r,r'\in\rad(f).\end{equation}
For $r,h\in\rad(f), l\geq 0$, we recall the differentiators in the universal enveloping algebra of  $\Der(\C_q)$:
$$\Omega_{r}^{(l,h)}:=\sum_{i=0}^l(-1)^i{l\choose i}e_{r-ih}e_{ih}.$$
These operators were introduced in \cite{BF2}.
\begin{lemma} \label{3.7}Let $M$ be an irreducible bounded $\Der(\C_q)$-module.
Then there exists an integer $l>1$ such that for all $r,h\in\rad(f)$,
the differentiator $\Omega_{r}^{(l,h)}$ annihilates $M$.
\end{lemma}
\begin{proof}
For any $n\in\Z^d$, the subspace $M(n):=\oplus_{r\in\rad(f)}M_{\a+n+r}$ is a bounded module over $W_u$.
Clearly $M(m)=M(n)$ for all $m,n\in \Z^d$ with $m-n\in\rad(f)$. For an $M(n)$, by Proposition 4.6 in \cite{BF1}, there exists  $K\in \N$ such that for all $r,h\in\rad(f)$ and $l>K$,
the differentiator $\Omega_{r}^{(l,h)}$ annihilates $M(n)$. Since the index of the subgroup $\rad(f)$ in $\Z^d$ is finite,
$M$ is a sum of a finite number of  $M(n)$.
 Thus there exists a large enough $l$ such that for all $r,h\in\rad(f)$,
the differentiator $\Omega_{r}^{(l,h)}$ annihilates $M$.
\end{proof}

\begin{theorem}Let $M$ be an irreducible bounded $\Der(\C_q)$-module such that $\C_q' M\neq 0$. Then
the $Z\mathcal{D}$-cover of $\widehat{M}$ is bounded.
\end{theorem}

\begin{proof} 
Let $\Delta$ be a complete
coset representatives  of the subgroup $\rad(f)$ in $\Z^d$ with $0\notin\Delta$. Clearly $\Delta$ is a finite set.
For a weight $\l\in \C^d$, the weight space $\widehat{M}_\l$ is spanned by
$$ \{\psi(t^{n+r}, M_{\l-n-r}): \  n\in\Delta, r\in\rad(f)\}.$$

We introduce a norm on $\rad(f)$:
$$\|r\|=\sum_{i=1}^d|\gamma_i|,$$
where $r=\sum_{i=1}^d\gamma_i\xi_i\in\rad(f)$, $\{\xi_1,\cdots,\xi_d\}$ is the $\Z$-basis of
$\rad(f)$ defined in (\ref{2.1}).
By Lemma \ref{3.7}, there exists an integer $l>1$ such that
the differentiator $\Omega_{r}^{(l,\xi_i)}$ annihilates $M$ for all $r\in\rad(f), i\in\{1,\dots,d\}$.

Let $S$ be the subspace of $\widehat{M}$ spanned by
$$ \psi(t^{n+r}, M_{\l-n-r}), \  n\in\Delta, r\in\rad(f)\ \text{with}\ \|r\|\leq\frac{ld}{2},$$
plus $\psi(t^{n_0+r_0}, M_{0})$ if $\lambda=n_0+r_0$ for some $n_0\in\Delta, r_0\in\rad(f)$.
Clearly $S$ is finite dimensional.

\

\noindent{\bf Claim:} $\widehat{M}_\l=S$.

In order to prove this claim, we only need to check  that $t^{n+r}\otimes M_{\l-n-r}$ belongs to
$S$ for any $n\in\Delta, r\in\rad(f)$. We use induction on $\|r\|$. If $|\gamma_i|\leq \frac{l}{2}$ for all
$i\in\{1,\dots,d\}$, then the claim is trivial. On the contrary, we assume that
$|\gamma_j|> \frac{l}{2}$ for some $j$. Without loss of generality, we assume that $\gamma_j>\frac{l}{2}$. The case
$\gamma_j<-\frac{l}{2}$ follows similarly. Clearly, the norms of $r-\xi_j,\dots,r-l\xi_j$ are strictly smaller than
$\|r\|$.
For $v\in M_{\l-n-r}$ with $\l-n-r\ne0$, since $e_0 v= (u \mid \l-n-r)v$, so we write $v=e_0w$ for some $w\in M_{\l-n-r}$.

From $0=\Omega_{r}^{(l,\xi_i)}t^n w=\sum_{i=0}^l(-1)^i{l\choose i}e_{r-i\xi_j} e_{i\xi_j}t^nw$, we see that
\begin{equation*} \sum_{i=0}^l(-1)^i{l\choose i}t^{n+r-i\xi_j}e_{i\xi_j}w+\sum_{i=0}^l(-1)^i{l\choose i}e_{r-i\xi_j} t^{n+i\xi_j}w =0,\end{equation*} where we have use that fact that $(u|n)\ne0$.
Note that  $ e_{r-i\xi_j} t^{n+i\xi_j}w=(u\mid n+i \xi_j)t^{n+r}w+t^{n+i\xi_j}e_{r-i\xi_j}w$. From
$$\sum_{i=0}^l(-1)^i{l\choose i}=\sum_{i=0}^l(-1)^i i{l\choose i}=0,$$ we get that
\begin{equation*} \sum_{i=0}^l(-1)^i{l\choose i}t^{n+r-i\xi_j}e_{i\xi_j}w+\sum_{i=0}^l(-1)^i{l\choose i}t^{n+i\xi_j}e_{r-i\xi_j} w =0.\end{equation*}
Thus $$t^{n+r}v=-\sum_{i=1}^l(-1)^i{l\choose i}t^{n+r-i\xi_j}e_{i\xi_j}w-\sum_{i=0}^l(-1)^i{l\choose i}t^{n+i\xi_j}e_{r-i\xi_j} w,$$
i.e.,
\begin{equation}\label{3.3} \begin{split}\psi(t^{n+r},v)=& -\sum_{i=1}^l(-1)^i{l\choose i}\psi(t^{n+r-i\xi_j},e_{i\xi_j}w)\\
& -\sum_{k=0}^l(-1)^k{l\choose k}\psi(t^{n+k\xi_j},e_{r-k\xi_j} w).\end{split}\end{equation}
Note that $e_{i\xi_j}w\in M_{\l-n-(r-i\xi_j)}, e_{r-k\xi_j} w\in M_{\l-n-k\xi_j}$ and  $||r-i\xi_j||<\|r\|$ for any $i\in\{1,\cdots, l\}$, $\|k\xi_j\|\le\frac{ld}{2}$ for any $k\in\{0, 1,\cdots, l\}$, since $d\geq 2$.
By induction assumption the right hand side of (\ref{3.3}) belongs to $S$. Therefore the Claim is true. Hence
$\widehat{M}_\l$ is finite dimensional.
The theorem is proved.
\end{proof}

Now we are ready to prove our main theorem.

\

{\it Proof of Theorem 2.5.}
If $\C_q' M=0$, the module $M$ is an irreducible module over  $\W_d$. This case was proved in \cite{BF2} where $W$ is taken as a one dimensional $\gl_N$-module in the statement of the theorem.

Now we assume that $\C_q' M\neq0$.
By the irreducibility of $M$ and the fact that $\C_q' M$ is a submodule of $M$,  we see that $\C_q' M=M$.
Thus the homomorphism $\pi:\widehat{M}\to M$ is surjective.

From \cite{BF2} we know that each irreducible bounded weight $W_d$-module has a support of the form $\alpha+\rad(f)$ for some $\alpha\in\C^d$ (possibly $0$ may be removed from this coset).  Since $[\Z^d:\rad(f)]<\infty$, then $\widehat{M}$ has a composition series of $Z\mathcal{D}$-submodules:
$$0=\widehat{M}_0\subset \widehat{M}_1\subset \cdots \subset \widehat{M}_s=\widehat{M}.$$
 Thus each quotient $\widehat{M}_i/\widehat{M}_{i-1}$
is an irreducible $Z\mathcal{D}$-module. Let $k$ be the smallest integer such that $\pi(\widehat{M}_k)\neq 0$.
By the irreducibility of $M$, we see that  $\pi(\widehat{M}_k)=M$ and $ \pi(\widehat{M}_{k-1})=0$. Thus we have a surjective $\Der(\C_q)$-module homomorphism from $\widehat{M}_k/\widehat{M}_{k-1}$ to $M$. By Theorem 4.4 in \cite{LZ1}, we know that $\widehat{M}_k/\widehat{M}_{k-1}$ is isomorphic to
$\mathcal{V}^\alpha(V,W)$ for some
finite dimensional irreducible $\gl_d$-module $V$,  finite
dimensional $\Gamma$-graded-irreducible
 $\gl_N$-module $W$, and $\alpha\in \mathbb{C}^d$. This completes the proof.
\qed


\begin{center}

\end{center}



\vspace{3mm}

\noindent G.L.: School of Mathematics and Statistics, Henan University, Kaifeng 475004, China. Email:
liugenqiang@amss.ac.cn

\vspace{0.2cm} \noindent K.Z.: Department of Mathematics, Wilfrid
Laurier University, Waterloo, ON, Canada N2L 3C5,  and Department of
Mathematics, Xinyang Normal University, Xinyang, Henan,  464000,
P.R. China. \\
 Email:
kzhao@wlu.ca

\end{document}